\documentclass[12pt,reqno]{amsart}
\usepackage{graphicx}
\usepackage{amsmath,amsthm,amsfonts,amscd,amssymb,mathrsfs}
\usepackage[all]{xypic}
\usepackage[enableskew]{youngtab} 

\usepackage{etoolbox}

\usepackage{color}
\usepackage{arydshln}
\usepackage{latexsym}
\usepackage[all]{xy}
\usepackage{cite}
\numberwithin{equation}{section}
\newtheorem{theorem}{Theorem}[section]

\newtheorem{lemma}[theorem]{Lemma}

\newtheorem{ex}[theorem]{Example}
\newtheorem*{ex*}{Example}
\newtheorem{df}[theorem]{Definition}

\newcommand{\rank}{\mbox{rank}\,}

\newcommand\g{\mathfrak g}
\newcommand\s{\mathfrak s}

\newcommand\gl{\mathfrak {gl}}

\renewcommand\l{\mathfrak l}

\renewcommand\s{\mathfrak s}

\newcommand\N{\mathcal N}

\newcommand\F{\mathbb F}

\newcommand\ad{\operatorname{ad}\,}

\newcommand\End{\mbox{\rm End\,}}
\newcommand\Hom{\mbox{\rm Hom\,}}
\newcommand\Der{\mbox{\rm Der\,}}

\newcommand{\Ker}{\mbox{\rm Ker\,}}
\renewcommand{\Im}{\mbox{\rm Im\,}}

\newcommand{\cen}[1]{{\rm Z}({#1})}

\newcommand{\mc}{\mathcal }

\newcommand{\E}[2]{E^{(#2)}_{#1}}

\def\keywords#1{{\bf Keywords:}{#1}}

\begin{document}
\title[Derivations of Lie algebra of strictly block upper triangular matrices]{Derivations of the Lie algebra of strictly block upper triangular matrices}
\author{Prakash Ghimire, Huajun Huang}
\date{}
\subjclass[2010]{Primary 17B40, Secondary 15B99, 16W25, 17B05}

\maketitle
\begin{abstract}
Let  $\N$ be the Lie algebra  of all $n\times n$ strictly block upper triangular matrices  over a field $\F$  relative to a given partition.
In this paper, we give an explicit description of all derivations of $\N$.
\end{abstract}

\keywords{{\bf Keywords:} derivation; strictly block upper triangular matrix; nilpotent Lie algebra.}

\section{Introduction}

Given a field $\F$, let $M_{mn}$ be the set of all $m\times n$ matrices over $\F$, and $M_{n}:= M_{nn}$. Let $\N$ denote the set of all strictly block upper triangular matrices in $M_{n}$ relative to a given partition. Then $\N$ can be viewed as a Lie subalgebra of $M_{n}$ with  the standard Lie bracket $[X, Y]= XY-YX$.

The main purpose of this article is to study the derivations of  Lie algebra $\N$. Recall that  a derivation of Lie algebra $\g$ is an $\F$-linear map $f:\g \to \g$  that satisfies
\begin{eqnarray*}
f([X, Y])= [f(X), Y]+ [X, f(Y)]\qquad \ \ \text{for all}\ X, Y\in \g.
\end{eqnarray*}
The result could be viewed as an extension (over $\F$) of  Ou, Wang and Yao's work on the derivations of the Lie algebra of strictly upper triangular matrices in $\g\l(n, R)$, where $R$ is a commutative ring with unity \cite{OWY}.

The derivation algebras $\Der(\N)$ for all possible $\N$ constitute an important family of the derivation algebras of nilpotent Lie algebras. Each $\N$ is a direct sum of root spaces in the root space decomposition of $\s\l(n,\F)$ or $\gl(n,\F)$, so that $\N$ and $\Der(\N)$ have elegant graded structures relative to the roots. 
Moreover, a derivation $f\in \Der(\g)$ of a Lie algebra $\g$ always maps the center $\cen{\g}$ into itself.  There is an induced Lie algebra homomorphism
$\phi: \Der(\g) \to \Der(\g/\cen{\g})$, $f\mapsto \bar{f}$, defined by
$$
\bar{f}(a+\cen{\g}):=f(a)+\cen{\g}\quad \text{ for } f\in\Der(\g),\ a\in\g.
$$
The  $\Ker{\phi}$ consists of all  $f\in\End(\g)$ that satisfy $\Ker{f}\supseteq [\g,\g]$ and $\Im{f}\subseteq \cen{\g}$. In particular, $\Ker{\phi}\simeq\Hom_{\F}(\g/[\g,\g],\cen{\g})$ as vector spaces.
Moreover, $\Der(\g)/\Ker{\phi}$ is isomorphic to a subalgebra of $\Der(\g/\cen{\g})$.
If $\g$ is  nilpotent, then $\g/\cen{\g}\simeq\ad \g$ is isomorphic to a subalgebra of the specific Lie algebra $\N$ of strictly upper triangular matrices in $\End(\g)$ \cite[Engel's Theorem]{JE}. Therefore, knowledge on   $\Der(\N)$ would be helpful to explore the   derivations of an arbitrary nilpotent Lie algebra.

In recent years, significant progress has been made in studying the derivations and generalized derivations of matrix Lie algebras and their subalgebras over a field or a ring.  Besides  Ou, Wang and Yao's work \cite{OWY},
Chen determined the structure of certain generalized derivations of a parabolic subalgebra of $\g\l(n, \F)$ over a field $\F$ with $\rm{char}(\F)\neq 2$ and $|\F|\geq n \geq 3$\cite{C}; Ghimire and Huang studied the derivations of Lie algebras of dominant upper triangular ladder matrices over  a field $\F$ with $\rm{char}(\F)\neq 2$\cite{GH}; Brice described the derivations of parabolic subalgebras of reductive Lie algebra over an algebraically closed and characteristics zero field, and proved the zero-product determined property of such derivation algebras\cite{B}. Let $R$ be a commutative ring with identity; Cheung characterized proper Lie derivations and gave  sufficient conditions for any Lie derivation to be proper for triangular algebras over $R$ \cite{CS}; Du and Wang investigated the Lie derivations of $2\times 2$ block generalized matrix algebras \cite{DW}; Wang, Ou and Yu described the derivations of intermediate Lie algebras between diagonal matrix algebra and upper triangular matrix algebra in $\g\l(n, R)$ \cite{WOY}; Wang and Yu characterized all derivations of parabolic subalgebras of $\g\l(n, R)$ \cite{WQ}.

The Lie triple derivations of matrix Lie algebra are also extensively studied over a ring $R$ with identity, for examples, on $\g\l(n, R)$ \cite{KZL}, on the algebra of upper triangular matrices of $\g\l(n, R)$ \cite{DB}, on the parabolic subalgebras of $\g\l(n, R)$ \cite{LCL} and on the algebra of strictly upper triangular matrices of $\g\l(n, R)$ \cite{TL}. More recently, Benkovi$\check{c}$ has described the Lie derivations and Lie triple derivations of upper triangular matrix algebras over a unital algebra \cite{DB.}.

In the rest of this paper, we fix a  given $t\times t$ partition of $M_n$, and the corresponding Lie algebra $\N$ of strictly block upper triangular matrices.
  Let $\N^{ij}$  denote the set of matrices in $M_n$  that take zero outside the $(i, j)$ block, and $\cen{\N}=\N^{1t}$  the center of $\N$.  Theorems \ref{thm:derivation not 2} and  \ref{thm:derivation char 2} provide the main results of this paper. Theorem \ref{thm:derivation not 2} shows that:
when $\rm{char}(\F)\neq 2$, every derivation $f$ of the Lie algebra $\N$ can be decomposed as follow:
\begin{eqnarray*}
f &=& \ad{X} +\varphi_{1t} + {\phi}^{12}_{2t}+ {\phi}^{t-1,t}_{1,t-1},
\end{eqnarray*}
where
\begin{itemize}
\item $X$ is a block upper triangular matrix in $M_{n}$ with the same partition as $\N$;
\item $\varphi_{1t}\in \End(\N)$ such that $\Im {\varphi_{1t}}\subseteq \cen{\N}$ and $\Ker{\varphi_{1t}}\supseteq [\N,\N]=\text{span}\{\N^{ij}\mid j>i+1\}$;
\item ${\phi}^{12}_{2t} \equiv 0$ unless  the $(1,2)$ block  of $\N$ has only one row, in which ${\phi}^{12}_{2t}\in \Der(\N)$ such that ${\phi}^{12}_{2t}(\N^{12})\subseteq \N^{2t}$  and ${\phi}^{12}_{2t}(\N^{ij})=0$ for the other $\N^{ij}\subseteq\N$; the explicit form of ${\phi}^{12}_{2t}$ is given in Lemma \ref{lemma A12.};
\item ${\phi}^{t-1,t}_{1,t-1}\equiv 0$ unless the $(t-1,t)$ block of $\N$ has only one column, in which  ${\phi}^{t-1,t}_{1,t-1}\in \Der(\N)$ such that ${\phi}^{t-1,t}_{1,t-1}(\N^{t-1,t})\subseteq \N^{1,t-1}$  and ${\phi}^{t-1,t}_{1,t-1}(\N^{ij})=0$ for the other $\N^{ij}\subseteq\N$; the explicit form of ${\phi}^{t-1,t}_{1,t-1}$ is given in Lemma  \ref{lemma A_t-1,t.}.
\end{itemize}
 Theorem \ref{thm:derivation char 2} shows that: when $\rm{char}(\F) = 2$, every derivation $f$  of $\N$ has the following form:
 $$ f=\ad{X} +\varphi_{1t} + {\phi}^{12}_{2t}+ {\phi}^{t-1,t}_{1,t-1} + {\psi}^{12;13}_{3t;2t}+\psi^{t-1,t;t-2,t}_{1,t-2;1,t-1},$$
 where  the  first four summand components  are the same as those in  $\rm{char}(\F)\neq 2$,
 and the last two  are determined as follow:
\begin{itemize}
\item ${\psi}^{12;13}_{3t;2t}\equiv 0$ unless the first block row of $\N$ has only one row, in which
  $ {\psi}^{12;13}_{3t;2t}\in \Der(\N)$ maps $\N^{12}$ to $\N^{3t}$, $\N^{13}$ to $\N^{2t}$, and the other $\N^{ij}\subseteq\N$ to 0;
the explicit form  of ${\psi}^{12;13}_{3t;2t}$ is given in Lemma  \ref{lemma A12 and A13.};

\item $\psi^{t-1,t;t-2,t}_{1,t-2;1,t-1}\equiv 0$ unless the last block column of $\N$ has only one column, in which
  $\psi^{t-1,t;t-2,t}_{1,t-2;1,t-1}\in \Der(\N)$ maps $\N^{t-1,t}$ to $\N^{1,t-2}$, $\N^{t-2,t}$ to $\N^{1,t-1}$, and the other $\N^{ij}\subseteq\N$ to 0;
the explicit form  of $\psi^{t-1,t;t-2,t}_{1,t-2;1,t-1}$ is given in Lemma  \ref{lemma A(t-1,t) and A(t-2,t)}.
\end{itemize}

In Section 2, we will prove several auxiliary results on  linear transformations of matrix spaces that satisfy certain special properties. In Section $3$, we will introduce and  characterize the derivations of   Lie algebra $\N$ over a field $\F$ with $\rm{char}(\F)\neq 2$. In Section $4$, we will explicitly describe the derivations of $\N$ for $\rm {char}(\F) = 2$.

\section{Preliminary} 

We prove some lemmas on linear transformations  between matrix spaces that have some special properties. These results will be useful in the proof of main theorem.

Denote $[n]:= \{1,2,\cdots, n\}$.
Let $E_{pq}^{(mn)}$ denote the $(p, q)$ standard matrix in $M_{mn}$ that has the only nonzero value $1$ in the $(p,q)$ entry.

\begin{lemma}\label{thm:block-tran-1}
If linear transformations $\phi:M_{mp}\to M_{mq}$ and $\varphi: M_{np}\to M_{nq}$ satisfy that
\begin{equation*}\label{block-tran-eq1}
\phi(AB)=A\varphi(B)\qquad\text{for all}\quad A\in M_{m n},\ \ B\in M_{n p},
\end{equation*}
then  there is $X\in M_{pq}$ such that
$\phi(C)=CX$ for $C\in M_{mp}$ and $\varphi(D)=DX$ for $D\in M_{np}$.
\end{lemma}

\begin{proof}
For any  $j\in [n]$ and $B\in M_{np}$,
$$\phi(\E{1j}{mn} B)=\E{1j}{mn}\varphi(B).$$
All such $\E{1j}{mn}B$ span the first row space of $M_{mp}$.   So $\phi$ sends the first row of $M_{mp}$ to the first row of $M_{mq}$.  There exists a unique $X\in M_{pq}$ such that
$$\E{1j}{mn}\varphi(B)=\phi(\E{1j}{mn}B)=\E{1j}{mn}BX,\qquad\text{for all \ } j\in[n],\ B\in M_{np}.$$
Therefore, $\varphi(B)=BX$.  Then $\phi(AB)=A\varphi(B)=ABX$ for any $A\in M_{mn}$ and $B\in M_{np}$.  All such $AB$ span $M_{mp}$.  So
$\phi(C)=CX$ for all $C\in M_{mp}$.
\end{proof}

\begin{lemma}\label{thm:block-tran-2}
 If linear transformations $\phi:M_{mp}\to M_{np}$ and $\varphi: M_{mq}\to M_{nq}$ satisfy that
\begin{equation*}\label{block-tran-eq2}
\phi(BA)=\varphi(B)A\qquad\text{for all}\quad  A\in M_{qp},\ B\in M_{mq},
\end{equation*}
then  there is $X\in M_{nm}$ such that
$\phi(C)=XC$ for $C\in M_{mp}$ and $\varphi(D)=XD$ for $D\in M_{mq}$.
\end{lemma}

\begin{proof}
The proof (omitted) is similar to that of  Lemma \ref{thm:block-tran-1}.
\end{proof}

\begin{lemma}\label{thm:block-tran-3}
 If linear transformations $\phi:M_{mp}\to M_{mq}$ and $\varphi: M_{qn}\to M_{pn}$ satisfy that
\begin{equation*}\label{block-tran-eq3}
\phi(A)B=A\varphi(B)\qquad\text{for all}\quad A\in M_{m p},\ \ B\in M_{qn},
\end{equation*}
then  there is $X\in M_{pq}$ such that
$\phi(C)=CX$ for $C\in M_{mp}$ and $\varphi(D)=XD$ for $D\in M_{qn}$.
\end{lemma}

\begin{proof}
For any $j\in [p]$ and  any $E_{kl}^{(qn)}\in M_{qn}$,
$$\phi(E_{1j}^{(mp)})E_{kl}^{(qn)}= E_{1j}^{(mp)}\varphi(E_{kl}^{(qn)}),$$
which shows that the only possibly nonzero row of $\phi(E_{1j}^{(mp)})$ is the first row. So $\phi$ maps the first row of $M_{mp}$
to the first row of $M_{mq}$. There exists a unique $X\in M_{pq}$ such that
$$E_{1j}^{(mp)}\varphi(E_{kl}^{(qn)})= \phi(E_{1j}^{(mp)})E_{kl}^{(qn)}=E_{1j}^{(mp)}XE_{kl}^{(qn)},\ \  \text{for all}\ j\in[p], \ \  E_{kl}^{(qn)}\in M_{qn}.$$
Therefore, $\varphi(E_{kl}^{(qn)})= XE_{kl}^{(qn)}$ for all $E_{kl}^{(qn)}\in M_{qn}$.
So $\varphi(B)= XB$ for $B\in M_{qn}$. Then $\phi(A)B= AXB$ for any $A\in M_{mp}$ and $B\in M_{qn}$.
Hence $\phi(A)= AX$ for all $A\in M_{mp}$.
\end{proof}

\begin{lemma}\label{thm:block-tran-4}
If linear transformations $f:M_{pr}\to M_{pr}$, $g:M_{pq}\to M_{pq}$, and $h:M_{qr}\to M_{qr}$ satisfy that
\begin{equation}\label{block-tran-eq4}
f(AB)=g(A)B+Ah(B) \quad \text{for all}\quad A\in M_{pq}, \ \ B\in M_{qr},
\end{equation}
then there exist $X\in M_{p}, Y\in M_{r}, Z\in M_{q}$  such that
\begin{eqnarray}
f(C)&=& XC+CY \qquad \text{for \ } C\in M_{pr},
\label{block-tran-eq4.1}
\\
g(A) &=& XA+AZ \qquad \text{for \ }  A\in M_{pq},
\label{block-tran-eq4.1'}
\\
h(B)&=& BY-ZB   \qquad    \text{for \ } B\in M_{qr}.
\label{block-tran-eq4.1''}
\end{eqnarray}
\end{lemma}

\begin{proof}
For any $n\in [p]$, $j, k\in [q]$, $m\in [r]$, $E^{(pq)}_{nj}\in M_{pq}$ and $E^{(qr)}_{km}\in M_{qr}$,
\begin{eqnarray}\label{block-tran-eq a}
f(E^{(pq)}_{nj}E^{(qr)}_{km})= g(E^{(pq)}_{nj})E^{(qr)}_{km}+ E^{(pq)}_{nj}h(E^{(qr)}_{km}).
\end{eqnarray}
We further discuss \eqref{block-tran-eq a} in  two cases:
\begin{enumerate}
\item $j\neq k$: the left side of \eqref{block-tran-eq a} is zero and
\begin{eqnarray}\label{block-tran-eq b}
g(E^{(pq)}_{nj})E^{(qr)}_{km}=-E^{(pq)}_{nj}h(E^{(qr)}_{km}).
\end{eqnarray}
\item $j=k$: the left side of \eqref{block-tran-eq a} is $f(E^{(pr)}_{nm})$, and according to \eqref{block-tran-eq a}, the only possibly nonzero entries of $f(E^{(pr)}_{nm})$ are
\begin{eqnarray}
f(E^{(pr)}_{nm})_{im}&=&g(E^{(pq)}_{nk})_{ik}\qquad  \text{for all}\  \  i\in [p],\  i\neq n;
\label{block-tran-eq d}
\\
f(E^{(pr)}_{nm})_{n\ell}&=&h(E^{(qr)}_{km})_{k\ell}\qquad  \text{for all}\  \ \ell\in [r],\  \ell\neq m;
\label{block-tran-eq e}
\\
f(E^{(pr)}_{nm})_{nm}&=&g(E^{(pq)}_{nk})_{nk}+ h(E^{(qr)}_{km})_{km}.
\label{block-tran-eq f}
 \end{eqnarray}
\end{enumerate}

Next we define a linear transformation  $f':M_{pr} \to M_{pr}$  such that property \eqref{block-tran-eq4} still holds. For $C\in M_{pr}$, let
\begin{equation}
f'(C):= \left [\sum_{i,j\in[p]}f(E_{j1}^{(pr)})_{i1}E_{ij}^{(pp)}\right] C+ C\left[\sum_{k, \ell\in [r]}f(E_{1k}^{(pr)})_{1\ell}E_{k\ell}^{(rr)}\right]-f(E_{11}^{(pr)})_{11}C.\end{equation}
Then for any $n\in[p]$, $m\in[r]$ and $E_{nm}^{(pr)}\in M_{pr}$,
\begin{eqnarray*}\label{block-tran-eq4.1.1}
 f'(E_{nm}^{(pr)})&= &\sum_{i\in[p]}f(E_{n1}^{(pr)})_{i1}E_{im}^{(pr)} + \sum_{\ell\in [r]}f(E_{1m}^{(pr)})_{1\ell}E_{n\ell}^{(pr)}-f(E_{11}^{(pr)})_{11} E_{nm}^{(pr)},
 \end{eqnarray*}
which implies that the only possibly nonzero entries of $f'(E_{nm}^{(pr)})$ are
 \begin{eqnarray}
 f'(E_{nm}^{(pr)})_{im}& =&f(E_{n1}^{(pr)})_{i1}= f(E^{(pr)}_{nm})_{im} \  \qquad  \ \qquad \text{for}\ i\in [p],\ i\neq n,
 \label{block-tran-eq4.4}
 \\
 f'(E_{nm}^{(pr)})_{n\ell}&=&f(E_{1m}^{(pr)})_{1\ell}= f(E^{(pr)}_{nm})_{n\ell} \ \ \qquad  \ \qquad \text{for}\ \ell\in [r], \ \ell\neq m,
 \label{block-tran-eq4.4.1}
 \\
 f'(E_{nm}^{(pr)})_{nm}&=& f(E_{n1}^{(pr)})_{n1}+ f(E_{1m}^{(pr)})_{1m}- f(E_{11}^{(pr)})_{11}= f(E^{(pr)}_{nm})_{nm},
 \label{block-tran-eq4.5}
 \end{eqnarray}
 where the last equality in \eqref{block-tran-eq4.4}, \eqref{block-tran-eq4.4.1} and \eqref{block-tran-eq4.5} is by \eqref{block-tran-eq d}, \eqref{block-tran-eq e} and \eqref{block-tran-eq f} respectively.
 Therefore, $f'=f$ on the standard  matrices of $M_{pr}$ and thus on the whole $M_{pr}$.
Denote
\begin{equation}
X:= \sum_{i,j\in[p]}f(E_{j1}^{(pr)})_{i1}E_{ij}^{(pp)}-f(E_{11}^{(pr)})_{11} I_{p},\qquad  Y:= \sum_{k, \ell\in [r]}f(E_{1k}^{(pr)})_{1\ell}E_{k\ell}^{(rr)}.
\end{equation}
We get $f(C) = f'(C)= XC + CY $ for   $C \in M_{pr}$. So  \eqref{block-tran-eq4.1} is done.
Now for $A\in M_{pq}$ and $B\in M_{qr}$,  by \eqref{block-tran-eq4},
 \begin{eqnarray*}
 g(A)B + Ah(B)=f(AB)= XAB+ABY
 \Longrightarrow (g(A)-XA)B= A(BY-h(B)).
 \end{eqnarray*}
 Applying Lemma \ref{thm:block-tran-3} to  $\phi:M_{pq}\to M_{pq}$ defined by $\phi(A)= g(A)-XA$ and $\varphi: M_{qr}\to M_{qr}$ defined by $\varphi(B)= BY-h(B)$, we will find $Z\in M_{q}$ such that
 \begin{eqnarray*}
 g(A)-XA &=& \phi(A)=AZ\qquad\text{ for } A\in M_{pq},
 \\
 BY-h(B) &=& \varphi(B)= ZB\qquad\text{ for } B\in M_{qr},
 \end{eqnarray*}
which imply \eqref{block-tran-eq4.1'} and \eqref{block-tran-eq4.1''}.
\end{proof}

\section{Derivations of $\N$ for ${\rm char}(\F)\ne 2$}

In this section, we will give an explicit description of $\Der (\N)$ for the Lie algebra $\N$ over a field $\F$ with ${\rm char}(\F)\ne 2$.
We first introduce some notations.

\begin{df} Suppose $\N$ and the conformal partition of $M_n$ are already given.
\begin{enumerate}
\item Let $\mathcal{B}$ denote the set of all block upper triangular matrices in $M_{n}$.
\item Let ${N}_{ij}$ denote the set of all submatrices in the $(i, j)$ block of $M_{n}$.  The $(i,j)$ block of a matrix $A\in M_n$ is denoted by $A_{ij}$ or $(A)_{ij}$.
If $A\in M_n$ is not given, $A_{ij}$ may also refer to an arbitrary matrix in $N_{ij}$.
\item Let ${\N}^{ij}$ denote the subset of $M_{n}$ consisting of matrices that take zero outside of the $(i,j)$ block.
For a matrix $B\in N_{ij}$, let $B^{ij}$ denote the embedding of $B$ into  $\N^{ij}$ by placing $B$ on the (i,j) block.
If $B\in N_{ij}$ is not given, $B^{ij}$ may also refer to an arbitrary matrix in $\N^{ij}$.
\end{enumerate}
A notation of double index, say $N_{ij}$, may be written as $N_{i,j}$ (resp. $\N^{ij}$ as $\N^{i,j}$) for clarity purpose.
\end{df}

The Lie algebra ${\mathcal B}$ is the normalizer of $\N$ in $M_n$.  For any $X\in{\mathcal B}$, the  adjoint action
$$\ad X:\N\to\N,\qquad Y\mapsto [X,Y],
$$
is a derivation of $\N$. Now we state our main theorem of this section.

\begin{theorem} \label{thm:derivation not 2}
Suppose $\rm{char}(\F)\neq 2$. Then every derivation $f$  of the Lie algebra $\N$ can be  (not uniquely) written as
\begin{eqnarray}\label{main result}
f &=& \ad{X} +\varphi_{1t} + {\phi}^{12}_{2t}+ {\phi}^{t-1,t}_{1,t-1},
\end{eqnarray}
where  the summand components are described below:
\begin{enumerate}
\item  $X\in {\mc B}$ is a block upper triangular matrix conformal to $\N$.

\item \label{def of rho} $\varphi_{1t}\in \End(\N)$ satisfies that
 $\Ker{\varphi_{1t}}$ contains $\displaystyle [\N,\N]=\bigoplus_{1<i+1<j\le t}\N^{ij}$, and
 $\Im{\varphi_{1t}}$ is contained in the center $\cen{\N}=\N^{1t}$ of $\N$.

\item ${\phi}^{12}_{2t} \equiv 0$ unless  the $(1,2)$ block  of $\N$ has only one row, in which ${\phi}^{12}_{2t}\in \Der(\N)$ such that ${\phi}^{12}_{2t}(\N^{12})\subseteq \N^{2t}$  and ${\phi}^{12}_{2t}(\N^{ij})=0$ for the other $\N^{ij}\subseteq\N$; the explicit form of ${\phi}^{12}_{2t}$ is given in Lemma \ref{lemma A12.};

\item ${\phi}^{t-1,t}_{1,t-1}\equiv 0$ unless the $(t-1,t)$ block of $\N$ has only one column, in which  ${\phi}^{t-1,t}_{1,t-1}\in \Der(\N)$ such that ${\phi}^{t-1,t}_{1,t-1}(\N^{t-1,t})\subseteq \N^{1,t-1}$  and ${\phi}^{t-1,t}_{1,t-1}(\N^{ij})=0$ for the other $\N^{ij}\subseteq\N$; the explicit form of ${\phi}^{t-1,t}_{1,t-1}$ is given in Lemma  \ref{lemma A_t-1,t.}.

\end{enumerate}

\end{theorem}

The cases $t=1$ and $t=2$ are trivial.  So we assume $t\geq 3$ in the following discussion.
Before proving the main theorem, we present several auxiliary results on the images $f(\N^{ij})$ for $f\in\Der(\N)$ and $\N^{ij}\subseteq\N$.
The next lemma concerns the range of $f$ on superdiagonal blocks of $\N$ except for $\N^{12}$ and $\N^{t-1,t}$.

\begin{lemma}\label{thm:f-N-k-k+1}
For $f\in \Der({\N})$ and $1<k<t-1$,
\begin{equation}\label{eq for (k,k+1)}
f(\N^{k,k+1}) \subseteq \sum_{p=1}^{k-1}{\N^{p,k+1}}+\sum_{q=k+1}^{t}\N^{kq}+\cen{\N}.
\end{equation}
In other words, the image  $f(\N^{k,k+1})$ is located on the $k$-th block row and the $(k+1)$-th block column of $\N$ as well as in the center $\cen{\N}=\N^{1t}$.
\end{lemma}

\begin{proof}
Given any $ A^{k,k+1}\in \N^{k,k+1}$, it suffices to prove that $f(A^{k,k+1})_{ij}=0$ for $i<j$, $i \neq k$, $j\neq k+1$, and $(i,j)\neq (1,t)$. Either $i>1$ or $j<t$. Without loss of generality, suppose $j<t$ (similarly for $i>1$). Then for any $A^{jt}\in\N^{jt}$, the $(i,t)$ block of $f([A^{k,k+1}, A^{jt}])$ is
\begin{align*}
0=f([A^{k,k+1}, A^{jt}])_{it}&=[f(A^{k,k+1}), A^{jt}]_{it}+[A^{k,k+1}, f(A^{jt})]_{it}
= f(A^{k,k+1})_{ij}(A^{jt})_{jt}
\end{align*}
where the last equality is by the assumptions on $i, j$.
Therefore   $f(A^{k,k+1})_{ij}=0$.
\end{proof}

Now consider the range of $f$  on $\N^{12}$ and $\N^{t,t-1}$ for $f\in\Der(\N)$.
The case ${\rm char}(\F)\ne 2$  would be simpler in the following lemma.

\begin{lemma}\label{image of (12)}
Let $f\in \Der(\N)$.
Then  the image $f(\N^{12})$ is located on the first block row, the $(2,t)$ block, and the $(3,t)$ block of $\N$; and  $f(\N^{t-1,t})$ is  on the last block column, the $(1,t-1)$ block, and the $(1,t-2)$ block of $\N$:
  \begin{eqnarray}\label{eq for (12)}
  f(\N^{12})&\subseteq & \sum_{q=2}^{t}{\N}^{1q}+{\N}^{2t}+{\N}^{3t},
\\
\label{eq for (t-1,t)}
  f(\N^{t-1,t}) &\subseteq & \sum_{p=1}^{t-1} {\N}^{pt}+{\N}^{1,t-1}+{\N}^{1,t-2}.
  \end{eqnarray}
Furthermore, when ${\rm char}(\F)\ne 2$, the $(3,t)$ block of  $f(\N^{12})$ and the $(1,t-2)$ block of $f(\N^{t-1,t})$ are zero.
\end{lemma}

\begin{proof}
 The case $t=3$ is obviously true. We now assume that $t\geq 4$.

To get \eqref{eq for (12)},  we prove that $f(A^{12})_{ij}=0$ for any $A^{12}\in\N^{12}$, $1<i<j$, and $(i,j)\notin\{(2,t), (3,t)\}$. Either $i>3$ or $j<t$.
\begin{enumerate}
  \item Suppose $j<t$. Then for any  $A^{jt}\in {\N}^{jt}$,
\begin{align*}
0=f([A^{12}, A^{jt}])_{it}=[f(A^{12}), A^{jt}]_{it}+[A^{12}, f(A^{jt})]_{it}.
\end{align*}
Therefore, $0=f(A^{12})_{ij}(A^{jt})_{jt}$, and thus $f(A^{12})_{ij}=0$.
  \item Suppose $3<i$. Then for any $A^{3i}\in {\N}^{3i}$,
  \begin{align*}
0=f([A^{12}, A^{3i}])_{3j}=[f({A^{12}}), {A^{3i}}]_{3j}+[{A^{12}}, f({A^{3i}})]_{3j}.
\end{align*}
Therefore, $0=-(A^{3i})_{3i}f(A^{12})_{ij}$, which implies that $f({A^{12}})_{ij}=0$.
\end{enumerate}
Overall, \eqref{eq for (12)} is done.

Next, when  $\rm{char}({\mathbb F})\neq 2$,  we show that $f({A^{12}})_{3t}=0$. For any $A^{23}\in \N^{23}$,
\begin{eqnarray*}
0&=& f([{A^{12}}, [{A^{12}}, {A^{23}}]])_{1t}
\\
&=& [f({A^{12}}), [{A^{12}}, {A^{23}}]]_{1t}+[{A^{12}}, [f({A^{12}}), {A^{23}}]]_{1t}
 + [{A^{12}}, [{A^{12}}, f({A^{23}})]]_{1t}
\\
&=& -2 (A^{12})_{12}(A^{23})_{23} f({A^{12}})_{3t}.
\end{eqnarray*}
Since $\rm{char{\mathbb{(F)}}}\neq 2$, $0=(A^{12})_{12}(A^{23})_{23}f({A^{12}})_{3t}$.
Given $A^{12}$, the matrix $(A^{12})_{12}(A^{23})_{23}$ for any  $A^{23}\in \N^{23}$ could be any matrix in $N_{13}$ with rank no more than  $\rank A^{12}$.
Therefore $f({A^{12}})_{3t}=0$.

The proofs of \eqref{eq for (t-1,t)} and $f(\N^{t-1,t})_{1,t-2}=0$ when $\rm{char{\mathbb{(F)}}}\neq 2$ are similar.
\end{proof}

\begin{df}
Let $E^{ij}_{pq}$   denote the $(p,q)$ standard matrix in $\N^{ij}$. In other words, $E^{ij}_{pq}$ is the matrix in $M_n$ that has the only nonzero entry 1 in the $(p,q)$ position of the $(i,j)$ block of $M_n$.
\end{df}

The next two lemmas explicitly describe the $(2,t)$ block of $f(\N^{12})$ and the $(1,t-1)$ block of $f(\N^{t-1,t})$ for $f\in \Der(\N)$.

\begin{lemma}\label{lemma A12.}
Suppose  the $(1,2)$ block of $\N$ has the size $p\times q$.  For $f\in \Der(\N)$, the image $f(\N^{12})_{2t}$ has the following properties:
\begin{enumerate}
\item If $p\geq 2$, then $f(\N^{12})_{2t} = 0$.
\item If $p=1$, then $\{E^{12}_{1i}\mid i\in[q]\}$ is a basis of $\N^{12}$;
the $i$-th row of $f(E^{12}_{1j})_{2t}$ and the $j$-th row of $f(E^{12}_{1i})_{2t}$ are equal for any $i,j\in[q]$.
\end{enumerate}
Conversely, any   $f\in\End(\N)$ that satisfies $f(\N^{ij})=0$ for $\N^{ij}\subseteq\N$ with $(i,j)\ne (1,2)$, $f(\N^{12})\subseteq\N^{2t}$,
and the above hypothesis, is in $\Der(\N)$.
\end{lemma}

\begin{proof} Let $f\in \Der(\N)$.
\begin{enumerate}
\item When $p\geq 2$, it suffices to prove that $f(E^{12}_{ij})_{2t}=0$ for any
 $i\in [p], j\in [q]$. Since $p\geq 2$, we can choose
 $r\in [p]-\{i\}$. Then for any
$s\in [q]$, $E^{12}_{rs}\in {\N}^{12}$ and
\begin{eqnarray*}
0=f([E^{12}_{rs}, E^{12}_{ij}])_{1t}&=&[f(E^{12}_{rs}), E^{12}_{ij}]_{1t}
            +[E^{12}_{rs}, f(E^{12}_{ij})]_{1t}\\
            &=&-(E^{12}_{ij})_{12}f(E^{12}_{rs})_{2t}+ (E^{12}_{rs})_{12}f(E^{12}_{ij})_{2t}.
\end{eqnarray*}
Therefore, $(E^{12}_{rs})_{12}f(E^{12}_{ij})_{2t}=(E^{12}_{ij})_{12}f(E^{12}_{rs})_{2t}$.
Comparing the $r$-th rows on both sides, we see that
 the $s$-th row of $f(E^{12}_{ij})_{2t}$ is zero. Since
 $s\in [q]$ is arbitrary,   we have $f(E^{12}_{ij})_{2t}=0$.

\item
Suppose $p=1$.  The case $q=1$ is trivial. Now we assume that $q\geq 2$. For any $j\in [q]$, we can choose $i\in[q]-\{j\}$. Then
\begin{eqnarray*}
0=f([E^{12}_{1j},  E^{12}_{1i}])_{1t}&=&[f(E^{12}_{1j}), E^{12}_{1i}]_{1t}+[E^{12}_{1j}, f(E^{12}_{1i})]_{1t}
\\
&=&-(E^{12}_{1i})_{12}f(E^{12}_{1j})_{2t}+(E^{12}_{1j})_{12}f(E^{12}_{1i})_{2t}.
\end{eqnarray*}
Therefore,
\begin{eqnarray*}
(E^{12}_{1i})_{12}f(E^{12}_{1j})_{2t}=(E^{12}_{1j})_{12}f(E^{12}_{1i})_{2t}.
\end{eqnarray*}
Comparing the first rows, we see that the $i$-th row of $f(E^{12}_{1j})_{2t}$ is equals to the $j$-th row of $f(E^{12}_{1i})_{2t}$ for $i\neq j$.
\end{enumerate}

The last statement is easy to verify.
\end{proof}

\begin{lemma}\label{lemma A_t-1,t.}
Suppose the $(t-1,t)$ block of $\N$ has the size $p \times q$. For  $f\in \Der(\N)$,
the image $f(\N^{t-1,t})_{1,t-1}$ satisfies the following properties:
\begin{enumerate}
\item If $q\geq 2$, then
$f(\N^{t-1,t})_{1,t-1} = 0$.
\item If $q=1$, then $\{E^{t-1,t}_{i1}\mid i\in [p]\}$ is a basis of $\N^{t-1,t}$;
the $i$-th column of $f(E^{t-1,t}_{j1})_{1,t-1}$ and the $j$-th column of $f(E^{t-1,t}_{i1})_{1,t-1}$ are equal for any $i,j\in [p]$.
\end{enumerate}
Conversely, any  $f\in\End(\N)$ that satisfies $f(\N^{ij})=0$ for $\N^{ij}\subseteq\N$ and $(i,j)\ne (t-1,t)$, $f(\N^{t-1,t})\subseteq\N^{1,t-1}$,
and the above hyposthesis, is in $\Der(\N)$.
\end{lemma}

\begin{proof}
The proof (omitted) is similar to that of Lemma \ref{lemma A12.}.
\end{proof}

Next we consider the range of  $f$  on the other blocks of $\N$.

\begin{lemma}\label{image of (i,j) block}
 For $f\in \Der(\N)$, $i,j\in[t]$ and $j>i+1$, the image $f(\N^{ij})$   satisfies that:
 \begin{enumerate}
 \item
 If $\rm{char}({\mathbb{F}})\neq 2$, then
\begin{eqnarray}\label{eq for (ij)}
 f({\N^{ij}}) &\subseteq& \sum_{p=1}^{i-1}{{\mathcal{N}^{pj}}}+\sum_{q=j}^{t}{{\mathcal{N}^{iq}}}.
\end{eqnarray}

\item
If $\rm{char}({\mathbb{F}})= 2$, then \eqref{eq for (ij)} still holds for $(i,j)\not\in\{(1,3), (t-2,t)\}$, and
\begin{eqnarray}
\label{eq for (13)}
f(\N^{13})
&\subseteq&
\sum_{q=3}^{t} \N^{1q}+\N^{2t},
\\
\label{eq for (t-2,t)}
f(\N^{t-2,t})
&\subseteq&
\sum_{p=1}^{t-2}  \N^{pt} +\N^{1,t-1}.
\end{eqnarray}
\end{enumerate}

\end{lemma}

\begin{proof} First assume  ${\rm char}(\F)\ne 2$.
Let $j=i+k$, $k\geq 2$. We prove \eqref{eq for (ij)} by induction on $k$.
\begin{enumerate}
\item $k=2:$   $\N^{i,i+2}= \N^{i,i+1}\N^{i+1,i+2}=[\N^{i,i+1}, \N^{i+1, i+2}]$. For $A^{i,i+1}\in \N^{i,i+1}$ and $A^{i+1, i+2}\in \N^{i+1, i+2}$,
    \begin{eqnarray}
    \notag
    f([{A^{i,i+1}}, {A^{i+1, i+2}}])
    &=&
    [f({A^{i,i+1}}), {A^{i+1, i+2}}]+[{A^{i,i+1}}, f({A^{i+1, i+2}})]
    \\
    \label{f-N-i-i+2}
    &\in &
     \N^{i, i+2}+ \sum_{p=1}^{i-1}\N^{p,i+2}+\sum_{q=i+3}^{t}\N^{i,q}
    \end{eqnarray}
    where the last relation is by Lemmas \ref{thm:f-N-k-k+1} and \ref{image of (12)}.
    Thus $k=2$ is done.

\item $k=\ell>2:$ Suppose \eqref{eq for (ij)} holds for all $k<\ell$ where $\ell>2$ is given. Now $\N^{i, i+\ell}= \N^{i, i+2}\N^{i+2, i+\ell}=[\N^{i, i+2}, \N^{i+2, i+\ell}]$.
For any $A^{i,i+2}\in\N^{i,i+2}$ and $A^{i+2, i+\ell}\in\N^{i+2,i+\ell}$,
    \begin{eqnarray}
    \notag
    f([A^{i,i+2}, A^{i+2, i+\ell}])&=&
    [f({A^{i,i+2}}), {A^{i+2, i+\ell}}]+[{A^{i,i+2}}, f({A^{i+2, i+\ell}})]
    \\
    \label{f-N-i-i+ell}
    & \in & \N^{i, i+\ell} + \sum_{p=1}^{i-1} \N^{p,i+\ell}
    +\sum_{q=i+\ell+1}^{t}\N^{i,q}
    \end{eqnarray}
    where the last relation is by induction hypothesis, the case $k=2$, and  Lemmas \ref{thm:f-N-k-k+1} and \ref{image of (12)}. So \eqref{eq for (ij)} is true for $k=\ell$.
\item Overall, \eqref{eq for (ij)} is true for all $k$.
\end{enumerate}

Now consider the case ${\rm char}(\F)=2$. The relation \eqref{f-N-i-i+2} has two exceptions $i=1$ and $i=t-2$ from Lemma \ref{image of (12)}.  For $i=1$, by Lemmas \ref{thm:f-N-k-k+1} and \ref{image of (12)},
\begin{equation*}
f([A^{12},A^{23}])=[A^{12}, f(A^{23})]+[f(A^{12}),A^{23}]\subseteq \sum_{q=3}^{t} \N^{1q}+[\N^{2t}+\N^{3t},A^{23}]\subseteq \sum_{q=3}^{t} \N^{1q}+\N^{2t}.
\end{equation*}
We get \eqref{eq for (13)}.  Similarly, we can get \eqref{eq for (t-2,t)}.
The relation \eqref{f-N-i-i+ell} is unaffected by \eqref{eq for (13)} and \eqref{eq for (t-2,t)} when $i=1$ or $(i,\ell)=(t-4, 4)$.
So the induction can be proceeded for  ${\rm char}(\F)=2$.
\end{proof}

Now the range of $f\in\Der(\N)$ on $\N^{ij}$ is limited.
The next lemma explicitly describes the $f$-images  of  each $\N^{ij}$ in almost all nonzero blocks.
It essentially implies that the $f$-images on these blocks are the same as the images of an adjoint action of a block upper triangular matrix.
Denote the index set
\begin{equation}\label{Omega}
\Omega:=\{(p,q)\in [t]\times [t] \mid p<q\}- \{(1,t-1), (1,t), (2,t)\}.
\end{equation}

\begin{lemma}\label{lemma X^{pq}}
Let $f\in \Der{(\N)}$. Then for any $(p,q)\in\Omega$,
there exists $X_{pq}\in N_{pq}$ such that
\begin{align}
f({A^{ip}})_{iq} &=-(A^{ip})_{ip}X_{pq}\quad \text{for all \ } {A^{ip}}\in {\N^{ip}}\subseteq \N,
\label{N-A_ip}
\\
f({A^{qj}})_{pj} &=X_{pq}(A^{qj})_{qj}\quad \text{for all \ } {A^{qj}}\in {\N^{qj}}\subseteq \N.
\label{N-A_qj}
\end{align}
\end{lemma}

\begin{proof}
Given $p<q$ in $[t]$, we prove \eqref{N-A_ip} and \eqref{N-A_qj} by the following steps:
\begin{enumerate}
\item  We prove \eqref{N-A_qj} for $(q,j)=(t-1, t)$.  Then $1<p<t-1$. For  $A^{t-1, t}\in \N^{t-1, t}$ and $A^{1p}\in \N^{1p}$,
\begin{eqnarray*}
0&=&f([{A^{1p}}, {{A^{t-1, t}}}])_{1t}=[f({A^{1p}}), A^{t-1, t}]_{1t}+ [A^{1p}, f({A^{t-1,t}})]_{1t}
\\
 &=& f({A^{1p}})_{1,t-1}(A^{t-1,t})_{t-1,t}+(A^{1p})_{1p}f({A^{t-1,t}})_{pt}.
\end{eqnarray*}
Therefore, $$-f({A^{1p}})_{1,t-1}(A^{t-1,t})_{t-1,t}=(A^{1p})_{1p}f({A^{t-1,t}})_{pt}.$$
Applying Lemma \ref{thm:block-tran-3} to $\phi:N_{1p}\to N_{1,t-1}$ defined
by $\phi(C)=-f(C^{1p})_{1,t-1}$ and
$\varphi:N_{t-1,t}\to N_{pt}$ defined by $\varphi( D)=f(D^{t-1,t})_{pt}$,
we will find $X_{p,t-1}\in N_{p,t-1}$ such that
$f({A^{t-1,t}})_{pt}=X_{p,t-1}(A^{t-1,t})_{t-1,t}$ for all $A^{t-1,t}\in \N^{t-1,t}$.

\item Similarly, we can prove \eqref{N-A_ip} for $(i,p)=(1,2)$ via Lemma \ref{thm:block-tran-3}. In other words,
for $2<q<t$, there is $-Y_{2q}\in N_{2q}$
    such that $f({A^{12}})_{1q}=-(A^{12})_{12}Y_{2q}$ for all $A^{12}\in \N^{12}$.

\item
Now we prove \eqref{N-A_qj} for $(q, j)\neq (t-1,t)$.  Then $q<t-1$.
 Given any $j'>j$ in $[t]$, we have
  ${\N^{qj'}}={\N^{qj}}{\N^{jj'}}=[{\N^{qj}}, {\N^{jj'}}]$.
 Then for $A^{qj}\in\N^{qj}$ and $A^{jj'}\in\N^{jj'}$,
$$
f({A^{qj} A^{jj'}})_{pj'}
=f([{A^{qj}}, {A^{jj'}}])_{pj'}=[f({A^{qj}}), {A^{jj'}}]_{pj'}+[{A^{qj}}, f({A^{jj'}})]_{pj'}
=f({A^{qj}})_{pj}(A^{jj'})_{jj'}.
$$
Applying Lemma \ref{thm:block-tran-2} to $\phi:N_{qj'}\to N_{pj'}$ defined by $\phi(C)=f(C^{qj'})_{pj'}$ and
$\varphi:N_{qj}\to N_{pj}$ defined by $\varphi(D)=f({D}^{qj})_{pj}$, we will find $X_{pq}\in N_{pq}$ such that
$f({A^{qj}})_{pj}=X_{pq}(A^{qj})_{qj}$ for all $A^{qj}\in \N^{qj}$ and $(q,j)\neq (t-1,t)$.

\item
Similarly, we can prove \eqref{N-A_ip} for $(i,p)\ne (1,2)$ via  Lemma \ref{thm:block-tran-1}. In other words,
there exists $-Y_{pq}\in N_{pq}$ such that
$f({A^{ip}})_{iq}=-(A^{ip})_{ip} Y_{pq}$ for all $A^{ip}\in \N^{ip}$ and $(i,p)\neq (1,2)$.

\item
Finally, for any $ A^{ip}\in \N^{ip}, A^{qj}\in \N^{qj}$, we have $i<p<q<j$, $[{A^{ip}}, {A^{qj}}]=0$, so that
\begin{eqnarray*}
0&=&f([{A^{ip}}, {A^{qj}}])_{ij}
=[f({A^{ip}}), {A^{qj}}]_{ij}+[{A^{ip}}, f({A^{qj}})]_{ij}
\\
&=& f({A^{ip}})_{iq}(A^{qj})_{qj}+(A^{ip})_{ip}f({A^{qj}})_{pj}
\\
&=&-(A^{ip})_{ip}Y_{pq}(A^{qj})_{qj}+(A^{ip})_{ip}X_{pq}(A^{qj})_{qj}.
\end{eqnarray*}
Therefore, $X_{pq}=Y_{pq}$. The equalities \eqref{N-A_ip} and \eqref{N-A_qj} are proved. \qedhere
\end{enumerate}
\end{proof}

The next lemma concerns the derivations with image in the center of $\N$.

\begin{lemma}\label{map to center}
Suppose $f\in \End{\mathcal{(N)}}$ satisfies that
$$
f(\N)\subseteq \cen{\N}=\N^{1t},\qquad
\Ker{f}\supseteq [\N,\N]=\sum_{i,j\in[t], i+1<j} \N^{ij}.
$$
Then $f\in \Der(\N)$.
\end{lemma}

\begin{proof}
The $f$ satisfying the above conditions also satisfies the derivation property:
$$f([\mathcal{N}, \mathcal{N}])=0=[f (\mathcal{N}), \mathcal{N}]+[\mathcal{N}, f(\mathcal{N})].  \qquad\qquad \qedhere$$
\end{proof}

Now we are ready to prove Theorem \ref{thm:derivation not 2}.

\begin{proof}[Proof of Theorem \ref{thm:derivation not 2}]
By Lemma \ref{lemma X^{pq}},  for  $(p,q)\in\Omega$  we can find a matrix $X_{pq}\in N_{pq}$ that satisfies \eqref{N-A_ip} and \eqref{N-A_qj}.
Let $X^{pq}:=(X_{pq})^{pq}\in\N^{pq}$, and let
\begin{equation}
X_0:=\sum_{(p,q)\in\Omega} X^{pq}\ \in\ \N,\qquad f_0:=f-\ad X_0\ \in\ \Der(\N).
\end{equation}
The equalities \eqref{N-A_ip} and \eqref{N-A_qj} imply that
\begin{equation}\label{f_0 property}
f_0(\N^{ip})_{iq} =0\quad \text{for all \ }  \N^{ip}\subseteq \N,
\qquad
f_0(\N^{qj})_{pj} =0\quad \text{for all \ } \N^{qj}\subseteq \N.
\end{equation}
By \eqref{eq for (k,k+1)},  \eqref{eq for (ij)}, and Lemma \ref{image of (12)} for ${\rm char}(\F)\ne 2$, for any $\N^{ij}\subseteq\N$, the only possibly nonzero blocks of $f_0(\N^{ij})$ are
the $(i,j)$ block and  the following:
\begin{enumerate}
\item the $(1,t)$ block when $j=i+1$, and
\item the $(2,t)$ block when $(i,j)=(1,2)$, or
\item the $(1,t-1)$ block when $(i,j)=(t-1,t)$.
\end{enumerate}
Define ${\phi}^{12}_{2t},{\phi}^{t-1,t}_{1,t-1}\in\End(\N)$ such that for $A\in \N$,
\begin{eqnarray}
{\phi}^{12}_{2t}(A) &:=& f_0(A^{12})^{2t}=f(A^{12})^{2t},\\
{\phi}^{t-1,t}_{1,t-1}(A) &:=& f_0(A^{t-1,t})^{1,t-1}=f(A^{t-1,t})^{1,t-1}.
\end{eqnarray}
Then Lemmas \ref{lemma A12.} and \ref{lemma A_t-1,t.} show that ${\phi}^{12}_{2t},{\phi}^{t-1,t}_{1,t-1}\in\Der(\N)$.
We get a  derivation
\begin{eqnarray}\label{def of f1}
f_1:=f_0- {\phi}^{12}_{2t} - {\phi}^{t-1,t}_{1,t-1}= f- \ad {X_0}- {\phi}^{12}_{2t} - {\phi}^{t-1,t}_{1,t-1}.
\end{eqnarray}
 Define a linear map $\varphi_{1t}\in\End(\N)$ such that for $A\in \N$,
\begin{equation}
\varphi_{1t}(A):=\sum_{i=1}^{t-1} f_1(A^{i,i+1})^{1t}=\sum_{i=1}^{t-1} f (A^{i,i+1})^{1t}.
\end{equation}
Then Lemma \ref{map to center} implies that $\varphi_{1t}\in\Der(\N)$.
We get a new derivation
\begin{eqnarray}\label{def of f2}
f_2:=f_1-\varphi_{1t}= f- \ad {X_0}- {\phi}^{12}_{2t} - {\phi}^{t-1,t}_{1,t-1}-\varphi_{1t}
\end{eqnarray}
where $f_2(\N^{ij})\subseteq\N^{ij}$.

To get \eqref{main result}, it suffices to prove the following claim regarding $f_2$: there exist $X^{ii}\in \N^{ii}$ for $i\in [t]$ such that for each $k\in [t-1]$, the derivation $$f_2^{(k)} := f_2-\sum_{i=1}^{k+1} \ad X^{ii}$$
satisfies that $f_{2}^{(k)}(\N^{pq})=0$ for $1\leq p<q\leq k+1$. The proof is done by induction on $k$:
\begin{enumerate}
\item $k=1$: For any $A^{12}\in \N^{12}$ and $A^{23}\in \N^{23}$, we have
\begin{eqnarray}
f_2(A^{12}A^{23})_{13}&=&f_2([A^{12}, A^{23}])_{13}
=[f_2(A^{12}), A^{23}]_{13}+[A^{12}, f_2(A^{23})]_{13}
\notag \\
&=& f_2(A^{12})_{12}(A^{23})_{23}+ (A^{12})_{12}f_2(A^{23})_{23}.\label{A^12A^23}
\end{eqnarray}
By \eqref{block-tran-eq4.1'} in Lemma \ref{thm:block-tran-4}, there exist $X^{11}\in \N^{11}$ and $-X^{22}\in \N^{22}$  such that
$$
f_2(A^{12})_{12}= (X^{11}A^{12} - A^{12}X^{22})_{12}.
$$
Let $f_2^{(1)}:= f_2- \ad  X^{11} - \ad X^{22}$. Then $f_2^{(1)}(\N^{12}) =0$. The claim holds for $k=1$.

\item $k=2$: Applying \eqref{A^12A^23} to $f_2^{(1)}$:
$$f_2^{(1)}(A^{12}A^{23})_{13}= f_2^{(1)}(A^{12})_{12}(A^{23})_{23}+(A^{12})_{12}f_2^{(1)}(A^{23})_{23}=(A^{12})_{12}f_2^{(1)}(A^{23})_{23}.$$
By Lemma \ref{thm:block-tran-1}, there exists $-X^{33}\in \N^{33}$  such that $f_2^{(1)}(A^{13})_{13}= (-A^{13}X^{33})_{13}$ and $f_2^{(1)}(A^{23})_{23}=(-A^{23}X^{33})_{23}$. Let $f_2^{(2)}:=f_2^{(1)}-\ad X^{33}$. Then
$f_2^{(2)}(\N^{pq})=0$ for $1\le p<q\le 3$. So $k=2$ is done.

\item $k=\ell>2$: Suppose the claim holds for all $k<\ell$ where $\ell>2$ is given. In other words, there exist $X^{ii}\in \N^{ii}$ for all $i\in [\ell]$  such that
$f_2^{(\ell-1)}:= f_2-\sum_{i=1}^{\ell}\ad X^{ii}$ satisfies that $f_2^{(\ell-1)}(\N^{pq})=0$ for $1\leq p < q\leq \ell$. Similar to \eqref{A^12A^23}, for any $p\in [\ell-1]$,
$A^{p,\ell}\in \N^{p,\ell}$, $A^{\ell, \ell+1}\in \N^{\ell, \ell+1}$,
\begin{eqnarray*}
f_2^{(\ell-1)}(A^{p,\ell}A^{\ell, \ell+1})_{p,\ell+1}
&=& f_2^{(\ell-1)}(A^{p,\ell})_{p,\ell}(A^{\ell, \ell+1})_{\ell, \ell+1}+ (A^{p,\ell})_{p,\ell}f_2^{(\ell-1)}(A^{\ell, \ell+1})_{\ell, \ell+1}
\\
&=& (A^{p,\ell})_{p,\ell}f_2^{(\ell-1)}(A^{\ell, \ell+1})_{\ell, \ell+1}.
\end{eqnarray*}
By Lemma \ref{thm:block-tran-1}, there exists $-X^{\ell+1, \ell+1}\in \N^{\ell+1, \ell+1}$ such that
$$f_2^{(\ell-1)}(A^{p,\ell+1})_{p,\ell+1}=(-A^{p,\ell+1}X^{\ell+1, \ell+1})_{p,\ell+1} \qquad \text{for all \ } p\in[\ell].$$
Let $f_2^{(\ell)}:= f_2^{(\ell-1)}- \ad X^{\ell+1, \ell+1}$. Then $f_2^{(\ell)}(\N^{p, \ell+1})=0$ for $p\in[\ell]$.  So $k=\ell$ is proved.
 \end{enumerate}

 Overall,  the claim is completely proved; in particular, $f_2^{(t-1)}(\N)=0$.  Let $X:=X_0+\sum_{i=1}^{t} X^{ii}$, then we get
  \eqref{main result}.
\end{proof}

\section{Derivations of $\N$ for ${\rm char}(\F)= 2$}

When $\rm{char}(\F)=2$, $\Der(\N)$ is not completely described by Theorem \ref{thm:derivation not 2}.
Lemmas \ref{image of (12)}, \ref{image of (i,j) block}, and \cite[Section 2(D)]{OWY} motivate  us to construct the following example.

\begin{ex}\label{example}
Suppose $\rm{char}(\F)=2$. Let $\N$ consist of strictly upper triangular matrices in $M_{4}$.
So $\N$ has a basis ${\mathcal B} := \{E_{12}, E_{13}, E_{14},E_{23},E_{24}, E_{34}\}$, where $E_{ij}$
denotes the matrix in $M_{4}$ that has the only nonzero entry 1 in the $(i,j)$ position.
Define $f\in \End (\N)$ by $f(E_{12}) := -E_{34}$, $f(E_{13}) := E_{24}$,
and $f(E) := 0$ for all other matrices $E\in\mathcal B$. We prove that
\begin{eqnarray}\label{ex eq}
f([E, E']) = [f(E), E']+ [E, f(E')]
\end{eqnarray}
for any $E, E'\in \mathcal B$, so that $f\in \Der(\N)$.
There are  only two cases where one side of \eqref{ex eq} is possibly nonzero:
\begin{enumerate}
\item $\{E, E'\} = \{E_{12}, E_{23}\}$, where
$f([E, E']) = f(E_{13}) = E_{24}$, $[f(E), E']+ [E, f(E')] = E_{24}.$

\item $\{E, E'\} = \{E_{12}, E_{13}\}$, where
$f([E, E']) = 0$, $[f(E), E']+ [E, f(E')] = 2E_{14}=0.$
\end{enumerate}
Therefore,  $f \in \Der(\N)$.
However, $f$ can not be written as in \eqref{main result}.
\end{ex}

In Section 3,  Lemmas \ref{image of (12)} and \ref{image of (i,j) block} have special statements for ${\rm char}(\F)= 2$, while Lemmas \ref{thm:f-N-k-k+1}, \ref{lemma A12.}, \ref{lemma A_t-1,t.}, \ref{lemma X^{pq}}, \ref{map to center} remain  unchanged.
The following two lemmas completely describe the  images of a derivation on additional nonzero blocks when ${\rm char}(\F)=2$.

\begin{lemma}\label{lemma A12 and A13.}
When ${\rm char}(\F)= 2$, suppose  the $(1,2)$ block and the $(1,3)$ block of $\N$ has the size $m\times p$ and $m\times q$ respectively.
For $f\in \Der(\N)$, the images $f(\N^{12})_{3t}$ and $f(\N^{13})_{2t}$ satisfy that
\begin{enumerate}
\item If $m\geq 2$, then
\begin{eqnarray}\label{12 and 13. block of f(A) zero}
f(\N^{12})_{3t}=0,\qquad  f(\N^{13})_{2t} = 0.
\end{eqnarray}
\item If $m=1$, then
$\N^{12}$ has a basis $\{E^{12}_{1i}\mid i\in [p]\}$ and $\N^{13}$ has a basis $\{E^{13}_{1j}\mid j\in [q]\}$;
the $i$-th row of $f(E^{13}_{1j})_{2t}$ is equal  to  the $j$-th row of $f(E^{12}_{1i})_{3t}$
for any $i \in [p]$ and $j \in [q]$.
\end{enumerate}
Conversely, any $f\in\End(\N)$ that satisfies $f(\N^{ij})=0$ for $\N^{ij}\subseteq\N$ and $(i,j)\not\in\{(1,2),(1,3)\}$, $f(\N^{12})\subseteq\N^{3t}$,
$f(\N^{13})\subseteq\N^{2t}$, and the above hypothesis, is in $\Der(\N)$.
\end{lemma}

\begin{proof} Suppose $f\in\Der(\N)$.
Given $E^{12}_{ij}\in\N^{12}$ and $E^{13}_{rs}\in \N^{13}$,
\begin{equation*}
0=f([E^{12}_{ij}, E^{13}_{rs}])_{1t}=[E^{12}_{ij}, f(E^{13}_{rs})]_{1t}+[f(E^{12}_{ij}), E^{13}_{rs}]_{1t}
            =(E^{12}_{ij})_{12}f(E^{13}_{rs})_{2t} - (E^{13}_{rs})_{13}f(E^{12}_{ij})_{3t}.
\end{equation*}
Therefore,
\begin{equation}\label{char-2-12-13-blocks}
(E^{12}_{ij})_{12}f(E^{13}_{rs})_{2t}=(E^{13}_{rs})_{13}f(E^{12}_{ij})_{3t}.
\end{equation}
\begin{enumerate}
\item If $m\ge 2$, then for a fixed $E^{12}_{ij}\in\N^{12}$, we can choose $r\in [m]-\{i\}$.
Comparing the $r$-th rows in the equality \eqref{char-2-12-13-blocks}, we see that
the $s$-th row of $f(E^{12}_{ij})_{3t}$ is zero.  Then $f(E^{12}_{ij})_{3t}=0$ since $s$ is arbitrary.
Similarly, $f(E^{13}_{rs})_{2t}=0$.  We get \eqref{12 and 13. block of f(A) zero}.

\item If $m=1$, then $i=r=1$ in \eqref{char-2-12-13-blocks},
which implies that the $j$-th row of $f(E^{13}_{1s})_{2t}$ is equal to the $s$-th row of $f(E^{12}_{1j})_{3t}$.
\end{enumerate}

Conversely, suppose $f\in\End(\N)$  satisfies that $f(\N^{ij})=0$ for $\N^{ij}\subseteq\N$ and $(i,j)\not\in\{(1,2),(1,3)\}$, $f(\N^{12})\subseteq\N^{3t}$,
$f(\N^{13})\subseteq\N^{2t}$, and the hypothesis in Lemma \ref{lemma A12 and A13.} (1) or (2).  When $m\ge 2$,
$f\equiv 0$; when $m=1$,  $f$ satisfies \eqref{char-2-12-13-blocks} for $i=r=1$ and all  $j\in[p]$, $s\in[q]$.
In both cases, $f$ satisfies the derivation property and thus $f\in\Der(\N)$.
\end{proof}

\begin{lemma}\label{lemma A(t-1,t) and A(t-2,t)}
When ${\rm char}(\F)= 2$, suppose  the $(t-1,t)$ block and the $(t-2,t)$ block of $\N$ has the size $p\times m$ and $q\times m$ respectively.
For $f\in \Der(\N)$, the images $f(\N^{t-1,t})_{1,t-2}$ and $f(\N^{t-2,t})_{1,t-1}$ satisfy that
\begin{enumerate}
\item If $m\geq 2$, then
\begin{eqnarray}\label{(t-1,t) and (t-2,t). block of f(A) zero}
f(\N^{t-1,t})_{1,t-2}=0,\qquad  f(\N^{t-2,t})_{1,t-1} = 0.
\end{eqnarray}
\item If $m=1$, then
$\N^{t-1,t}$ has a basis $\{E^{t-1,t}_{i1}\mid i\in [p]\}$ and $\N^{t-2,t}$ has a basis $\{E^{t-2,t}_{j1}\mid j\in [q]\}$;
the $i$-th column of $f(E^{t-2,t}_{j1})_{1,t-1}$ is equal  to  the $j$-th column of $f(E^{t-1,t}_{i1})_{1,t-2}$
for any $i \in [p]$ and $j \in [q]$.
\end{enumerate}
Conversely, any $f\in\End(\N)$ that satisfies $f(\N^{ij})=0$ for $\N^{ij}\subseteq\N$ and $(i,j)\not\in\{(t-1,t),(t-2,t)\}$, $f(\N^{t-1,t})\subseteq\N^{1,t-2}$,
$f(\N^{t-2,t})\subseteq\N^{1,t-1}$, and the above hypothesis, is in $\Der(\N)$.
\end{lemma}

\begin{proof}
 The proof (omitted) is similar to that of Lemma \ref{lemma A12 and A13.}.
\end{proof}

Now we are able to describe  $\Der(\N)$ for the case ${\rm char}(\F)=2$.

\begin{theorem} \label{thm:derivation char 2}
When $\rm{char}(\F)=2$, every derivation $f$  of the Lie algebra $\N$ can be  (not uniquely) written as
\begin{eqnarray}\label{main result-char-2}
f &=& \ad{X} +\varphi_{1t} + {\phi}^{12}_{2t}+ {\phi}^{t-1,t}_{1,t-1} + {\psi}^{12;13}_{3t;2t}+\psi^{t-1,t;t-2,t}_{1,t-2;1,t-1},
\end{eqnarray}
where  the summand components
${X}$, $\varphi_{1t}$, ${\phi}^{12}_{2t}$, ${\phi}^{t-1,t}_{1,t-1}$
are described in Theorem \ref{thm:derivation not 2}, and
${\psi}^{12;13}_{3t;2t}$ and $\psi^{t-1,t;t-2,t}_{1,t-2;1,t-1}$ are determined as follow:
\begin{enumerate}
\item ${\psi}^{12;13}_{3t;2t}\equiv 0$ unless the first block row of $\N$ has only one row, in which
  $ {\psi}^{12;13}_{3t;2t}\in \Der(\N)$ maps $\N^{12}$ to $\N^{3t}$, $\N^{13}$ to $\N^{2t}$, and the other $\N^{ij}\subseteq\N$ to 0;
the explicit form  of ${\psi}^{12;13}_{3t;2t}$ is given in Lemma  \ref{lemma A12 and A13.};

\item $\psi^{t-1,t;t-2,t}_{1,t-2;1,t-1}\equiv 0$ unless the last block column of $\N$ has only one column, in which
  $\psi^{t-1,t;t-2,t}_{1,t-2;1,t-1}\in \Der(\N)$ maps $\N^{t-1,t}$ to $\N^{1,t-2}$, $\N^{t-2,t}$ to $\N^{1,t-1}$, and the other $\N^{ij}\subseteq\N$ to 0;
the explicit form  of $\psi^{t-1,t;t-2,t}_{1,t-2;1,t-1}$ is given in Lemma  \ref{lemma A(t-1,t) and A(t-2,t)}.

\end{enumerate}

\end{theorem}

\begin{proof}
Given $f\in\Der(\N)$, we can proceed the proof of Theorem \ref{thm:derivation not 2} up to \eqref{f_0 property}. Then we define 
${\psi}^{12;13}_{3t;2t}, \psi^{t-1,t;t-2,t}_{1,t-2;1,t-1}\in \End(\N)$ such that for $A\in \N$, 
\begin{eqnarray*}
{\psi}^{12;13}_{3t;2t}(A) &:=& f_0(A^{12})^{3t}+ f_0(A^{13})^{2t} = f(A^{12})^{3t}+ f(A^{13})^{2t},
\\
\psi^{t-1,t;t-2,t}_{1,t-2;1,t-1}(A) &:=& f_0(A^{t-1, t})^{1, t-2}+ f_0(A^{t-2, t})^{1, t-1} = f(A^{t-1, t})^{1, t-2}+ f(A^{t-2, t})^{1, t-1}.
\end{eqnarray*}
Both linear maps are derivations by Lemmas \ref{lemma A12 and A13.} and \ref{lemma A(t-1,t) and A(t-2,t)}.
Subtracting them from $f$, we can continue the remaining proof of Theorem \ref{thm:derivation not 2}.
\end{proof}

\end{document}